\documentclass[12pt]{article} 
\usepackage[T2A]{fontenc} 
\usepackage[utf8]{inputenc} 
\usepackage{comment, amsthm, amssymb, amsmath, mathrsfs, bbm} 
\usepackage[english]{babel}
\newtheorem{conjecture}{Conjecture}

\newtheorem{theorem}{Theorem} 
\newtheorem{lemma}{Lemma}[section]

\theoremstyle{remark}
\newtheorem*{remark}{Remark}

\usepackage{hyperref}

\newtheorem{proposition}{Proposition}

\newtheorem{rem}{Remark}[theorem]

\theoremstyle{definition}
\theoremstyle{plain}
\newtheorem{definition}{Definition}[section]
\renewcommand{\leq}{\leqslant}
\renewcommand{\geq}{\geqslant}
\renewcommand{\le}{\leqslant}
\renewcommand{\ge}{\geqslant}
\DeclareMathOperator{\End}{End}
\def\T{\mathcal{T}}
\begin{document}
\newcommand{\eps}{\varepsilon}
\renewcommand{\phi}{\varphi}

\newcommand{\suchthat}{\, : \,}

\author{Dmitry Krachun, Fedor Petrov}
\title{On the size of $A+\lambda A$ for algebraic $\lambda$}
\maketitle
\begin{abstract}
For a finite set  $A\subset \mathbb{R}$ and real $\lambda$, let $A+\lambda A:=\{a+\lambda b :\,  a,b\in A\}$. Combining a structural theorem of Freiman on sets with small doubling constants together with a discrete analogue of Pr\'ekopa–Leindler inequality we prove a lower bound $|A+\sqrt{2}  A|\geq (1+\sqrt{2})^2|A|-O({|A|}^{1-\eps})$ which is essentially tight. We also formulate a conjecture about the value of $\liminf |A+\lambda A|/|A|$ for an arbitrary algebraic $\lambda$. Finally, we prove a tight lower bound on the Lebesgue measure of $K+\mathcal{T} K$ for a given linear operator $\mathcal{T}\in \operatorname{End}(\mathbb{R}^d)$
and a compact set $K\subset \mathbb{R}^d$ with fixed measure. This continuous result supports the conjecture and yields an upper bound in it. 
\end{abstract}
\section{Introduction}
Let $A$ be a finite non-empty set with elements in a commutative ring $R$ and let $\lambda$ be
an element of $R$. A number of papers
are devoted to bounding 
$|A+\lambda A|=|\{x+\lambda y:x,y\in A\}|$
in terms of $|A|$ and various generalizations
of this problem. 
In particular, the sums of several
dilates have been intensively studied.
The sums of the form
$A+\mathcal{L} A$ for linear maps $\mathcal{L}$
(when $R$ is a module over another ring)
were also considered in \cite{Mudgal2019}. 

In \cite{Konyagin2006} it is proved that
$$
|A+\lambda A|\ge C\frac{|A|\log |A|}{\log \log |A|}
$$
for an absolute constant $C$ and
any finite $A\subset \mathbb{R}$, $|A|>2$,
and any transcendental $\lambda\in \mathbb{R}$
(it is easy to see that for all transcendental $\lambda$ the minimal value of $|A+\lambda A|$
for fixed $|A|$ is the same).
This bound was improved in \cite{Sanders2008}
to $|A|\log^{4/3-o(1)} |A|$. By proving stronger versions of Freiman's theorem, it was further improved to $(\log{|A|})^{c\log\log{|A|}}|A|$ in \cite{schoen2011} and finally to $e^{\log^c{|A|}}|A|$ for some $c>0$ in \cite{sanders2012}. In the other direction, there exist arbitrary large sets $A$ such that $|A+\lambda A|\le e^{C\log^{1/2}{|A|}}|A|$ for an absolute constant $C>0$.

In \cite{BUKH2008} it is
proved among other bounds that $|A+3A|\geqslant 4|A|-O(1)$
for $A\subset \mathbb{R}$ and
$|\lambda_1 A+\ldots+\lambda_k A|\geqslant
(|\lambda_1| +\ldots+|\lambda_k|)|A|+o(|A|)$
for coprime integers $\lambda_1,\ldots,\lambda_k$. In
 \cite{Balog2014} it is
proved that $|A+\frac{p}{q}A|\geqslant (p+q)|A|-O(1)$
for a rational number $\frac pq$, where $p,q$
are coprime positive integers. 
In \cite{Chen2018} a bound 
\begin{equation*}\label{eq:Chen}
    |A+\lambda A|\geqslant (1+\lambda-\varepsilon)|A|
\end{equation*}
was proved for any fixed
real $\lambda\geqslant 1$ and $\varepsilon>0$
and large enough $A\subset \mathbb{R}$
(with bounds on $|A|$ depending on $\lambda$
and $\varepsilon$.)

Here we consider a somehow intermediate 
variant of the problem: an algebraic but 
not rational $\lambda$, namely, $\lambda=\sqrt{2}$.

Let $N,M$ be positive integers, consider the set 
$A=\{x+y\sqrt{2}:0\le x<N,0\le y<M\}$,
then $\sqrt{2} A\subset \{x+y\sqrt{2}:0\le x<2M,
0\le y<N\}$, therefore $A+\sqrt{2} A\subset \{x+y\sqrt{2}:0\le x<N+2M,
0\le y<M+N\}$ and 
$$
\frac{|A+\sqrt{2} A|}{|A|}
\le \frac{(N+2M)(M+N)}{MN}=2\frac{M}N+\frac{N}M+3
$$
that can be arbitrarily close to $3+2\sqrt{2}=(1+\sqrt{2})^2$ when
$N/M$ is close to $\sqrt{2}$. 

We prove that this constant is tight:

\begin{theorem}\label{main}
There exist absolute constants $C, \eps > 0$ such that 
\[
|A+\sqrt{2}A|\geq (1+\sqrt{2})^2|A| - C|A|^{1-\eps}.
\]
for every finite set $A\subset \mathbb{R}$.
\end{theorem}

The rest of the paper is organized as follows. In Section \ref{sec:reduction} we reduce Theorem \ref{main} to the case $A\subset \mathbb{Z}[\sqrt{2}]$, in Section \ref{sec:sumsets} we prove an inequality on sum of subsets of an abelian group which we later use. In Section \ref{sec:discrete-prekopa-leindler} we prove Theorem \ref{main} in the case when $A\subset \mathbb{Z}[\sqrt{2}]$ satisfies certain regularity condition, and in Section \ref{sec:freiman} we deduce Theorem \ref{main} for an arbitrary set $A$. Finally, in Section \ref{sec:continuous-setting} we prove a general analogue of the result in the continuous setting and state a conjecture on the value of
$\liminf |A+\alpha A|/|A|$ for an arbitrary
algebraic $\alpha$.

\section{Reduction to $\mathbb{Z}[\sqrt{2}]$}\label{sec:reduction}
We first prove a quite intuitive fact that for upper-bounding $|A+\sqrt{2}A|$ one may assume that all elements of $A$ are in $\mathbb{Z}[\sqrt{2}]$. We prove the following slightly more general fact.
\begin{lemma}\label{lm:reduction}
Suppose that $\alpha\in \mathbb{C}$ and $A$ is a finite set of complex numbers. Then there exists a finite set $B\subset \mathbb{Q}[\alpha]$ such that $|B|=|A|$ and $|B+\alpha \cdot B|\leq |A+\alpha \cdot A|$.
\end{lemma}
\begin{proof}
Let $V$ be a $\mathbb{Q}[\alpha]$ vector space generated by elements of $A$. 
There exists a linear functional $\varphi\colon
V\to \mathbb{Q}[\alpha]$ which
is injective on $A$ (a generic $\varphi$ works,
for example). Then for
$B=\varphi(A)$ we have $|B|=|A|$ and $$|B+\alpha\cdot B|
=|\varphi(A+\alpha\cdot A)|\leq |A+\alpha\cdot A|.$$
\end{proof}

\section{Sum of subsets of an abelian group}\label{sec:sumsets}

We need the following standard fact of
Plünnecke--Rusza type.

\begin{lemma}\label{lm:doubling}
Let $G$ be an abelian group. If sets $A, B\subset G$ with $|A|=|B|$ are such that $C:=A+B$ satisfies $|C|\leq K|A|$ then $|C+C|\leq K^6|C|$.
\end{lemma}
\begin{proof}
The case $A=\emptyset$ is clear, so
we suppose that $|A|=|B|>0$. 
By a variant of Plünnecke--Ruzsa inequality 
\cite[formula (2.4)]{Ruzsa1996}, we have $|A+A+A|\leq K^3 |A|$ and $|B+B+B|\leq K^3|B|$. We then use Ruzsa sum triangle inequality
\cite[formula (4.6)]{Ruzsa1996}: for any non-empty
subsets $X, Y, Z\subset G$ we have 
\[
|Y+Z| \leq \frac{|X+Y|\cdot |X+Z|}{|X|}.
\]
First, taking $X:=A, Y:=B, Z:=A+A$ we obtain 
\[
|A+A+B|\leq  \frac{|A+B|\cdot |A+A+A|}{|A|}\leq \frac{|C|\cdot K^3|A|}{|A|}=K^3|C|.
\]
Then, taking $X:=B, Y:=A+A, Z:=B+B$ we obtain
\[
|A+A+B+B|\leq  \frac{|A+A+B|\cdot |B+B+B|}{|B|}\leq \frac{K^3 |C|\cdot K^3|B|}{|B|}=K^6|C|.
\]
\end{proof}

\section{Discrete Pr\'ekopa--Leindler inequality} \label{sec:discrete-prekopa-leindler}
By Lemma \ref{lm:reduction} we may assume that $A\subset \mathbb{Z}[\sqrt{2}]$. Identifying a number $a+b\sqrt{2}$ with a point $(a,b)\in\mathbb{Z}^2$ we get $|A+\sqrt{2}A|=|A+\T A|$, where an operator $\T:\mathbb{Z}^2\rightarrow\mathbb{Z}^2$ is given by $\T (a, b) := (2b, a)$.

Note that for a compact set $\Omega\subset \mathbb{R}^2$ the inequality $|\Omega+\T \Omega|
\geqslant (1+\sqrt{2})^2|\Omega|$
(here $|\cdot|$ stands for Lebesgue measure)
follows from Brunn--Minkowski inequality. The idea
is to mimic a proof of Brunn--Minkowski. 
Among the various proofs we have chosen 
the one which uses Pr\'ekopa--Leindler inequality. It is well possible that others work, too. The discrete version of Pr\'ekopa--Leindler inequality
which we use is similar to that of 
IMO 2003 Shortlist Problem A6 \cite[section 3.44.2, problem 6]{dj}(proposed
by Reid Barton). 


%
\begin{definition}
For a set $A\subset \mathbb{Z}^2$ we define $\phi_x(A)$ (resp. $\phi_y(A)$) to be
the number of different $x$ (resp. y) coordinates of points in $A$.
\end{definition}
\begin{lemma}\label{lm:main-lemma}
Let $A\subset\mathbb{Z}^2$ be a finite set. Then one has
\begin{equation}\label{eq:lower-bound}
|A+\T A| \geq (1+\sqrt{2})^2|A|-60|A|^{1/2}-6\log{|A|}(\phi_x(A) + \phi_y(A)). 
\end{equation}
\end{lemma}
\begin{proof}
If $|A|=1$ this is trivial so we assume $|A|>1$. We induct on $|A|$ and for fixed $|A|$ we induct on the diameter of $A$. For a finite set $X\subset\mathbb{Z}^2$ let
\[
f(X):=(1+\sqrt{2})^2|X|-60|X|^{1/2}-6\log{|X|}(\phi_x(X) + \phi_y(X)).
\] 
We define a set $A_0$ (resp. $A_1$) to be a set of points of $A$ with even (resp. odd) abscissa. If one of the sets is empty (after a shift we may assume that $A_1$ is empty) then we can consider a set $A':=\T^{-1}A$ instead of $A$ which satisfies $f(A')=f(A)$ and has smaller diameter. So we may assume that both $A_0$ and $A_1$ are non-empty. By shifting $A$ if necessary, we may also assume that 
\[
|A_1|\geq |A|/2, \qquad |A_0|\leq |A|/2.
\]
Note that sets $A_0+\T A$ and $A_1+\T A$ are disjoint. So if $|A_0|\leq \frac{|A|}{(1+\sqrt{2})^2}$ we trivially have, using induction hypothesis applied to $A_1$, that 
\[
|A+\T A|\geq |A_1 + \T A_1| + |A_0+\T A| \geq f(A_1) + |A| \geq f(A).
\]
So from now on we assume that
\begin{equation}\label{eq:assumption-A0-A1}
\frac{|A|}{(1+\sqrt{2})^2} \leq |A_0|\leq \frac{|A|}{2},
\qquad
\frac{|A|}{2} \leq |A_1|\leq |A|-\frac{|A|}{(1+\sqrt{2})^2}. 
\end{equation}


We define three sets of variables indexed by integers:
\begin{align*}
x_i:=|A\cap \{(a, b)\in\mathbb{Z}^2:\, a=i\}|, \qquad 
&y_j:=|A\cap \{(c, d)\in\mathbb{Z}^2:\, d=j\}|, \\
z_k:=|(A+\T A)\cap \{(a, b)\in\mathbb{Z}^2:\, &a=k\}|.
\end{align*}
Recall that $A_0:=\{(a, b)\in A:\, 2|a\}$. We now want to estimate the size of $A_0+\T A$. We define  
\[
y=\max_i y_i, \qquad x^0=\max_{2 | i} x_i.
\]
If $x^0\geq (1+\sqrt{2})|A|^{1/2}$ then the set $A$ has at least this many points on the same vertical line, hence, $\T A$ has also as many points on one horizontal line which implies a lower bound $|A+\T A|\geq (1+\sqrt{2})^2|A|$ and \eqref{eq:lower-bound} trivially follows. Similarly, if $y\geq (1+\sqrt{2})|A|^{1/2}$, the set $\T A$ has at least this many point on a vertical line, hence, $A$ has at least this many points on a horizontal line and \eqref{eq:lower-bound} again trivially follows. So from now on we assume that $x^0, y\leq (1+\sqrt{2})|A|^{1/2}$.

Note that by Cauchy--Davenport theorem (which is trivial for the case of $\mathbb{Z}$) if $x_i, y_j > 0$ then $z_{i+2j}\geq x_i + y_j - 1$. This implies that for any $t$ one has 
\begin{align*}
&\{k:\, 2|k, z_k\geq t - 1\} \supset \{i:\, 2|i, x_i\geq \frac{x^0t}{x^0+y} \}+2\cdot \{j:\, y_j\geq \frac{yt}{x^0+y}\}.
\end{align*}

The lower bounds on the right are chosen such that either both sets we add are empty or both are not. Hence, we can apply Cauchy-Davenport again to obtain 
\begin{align*}
&|\{k:\, 2|k, z_k\geq t - 1\}| \geq |\{i:\, 2|i, x_i\geq \frac{x^0t}{x^0+y} \}|+\{j:\, y_j\geq \frac{yt}{x^0+y}\}|-1,\\
&|\{k:\, 2\not|k, z_k\geq t - 1\}| \geq  |\{i:\, 2\not|i, x_i\geq \frac{x^1t}{x^1+y} \}|+| \{j:\, y_j\geq \frac{yt}{x^1+y}\}|-1.
\end{align*}

If we then integrate the first inequality between $1$ and $x^0+y$ we get
\begin{align*}
|A_0&+\T A| = \sum_{2|k} z_k \geq \int_1^{x^0+y} |\{k:\, 2|k, z_k\geq t - 1\}|\, dt  
\\&\geq \left(\int_1^{x^0+y} |\{i:\, 2|i, x_i\geq \frac{x^0t}{x^0+y} \}|\, dt \right) + \left(\int_1^{x^0+y} |\{j:\, y_j\geq \frac{yt}{x^0+y} \}|\, dt \right) - (x^0+y-1) 
\\&=
\left(-|\{i:\, 2|i, x_i\geq 1\}|+\frac{x^0+y}{x^0}\cdot \sum_{2|i} x_i \right) + \left(-|\{j: y_j\geq 1\}|+\frac{x^0+y}{y}\cdot \sum_j y_j\right) - (x^0+y-1).
\end{align*} 
Using the bounds on $x^0, y$ that we assume we obtain
\[
|A_0+\T A|\geq \frac{x^0+y}{x^0}\cdot |A_0| + \frac{x^0+y}{y}\cdot |A| - \phi_x(A_0) - \phi_y(A) - (2+2\sqrt{2})|A|^{1/2}.
\]
Using that $|A_0|\leq |A|/2$, we obtain
\begin{align*}
|A_0+\T A|
& \geq
\left(\frac{x^0+y}{x^0} + \frac{2(x^0+y)}{y}\right)\cdot |A_0| - \phi_x(A_0) - \phi_y(A) - (2+2\sqrt{2})|A|^{1/2} 
\\&\geq 
(1+\sqrt{2})^2|A_0| - \phi_x(A)-\phi_y(A) - (2+2\sqrt{2})|A|^{1/2}.
\end{align*}
Also, by induction hypothesis, we have 
\[
|A_1+\T A|\geq |A_1+\T A_1|\geq (1+\sqrt{2})^2|A_1| - 60|A_1|^{1/2}-6\log{|A_1|}(\phi_x(A_1)+\phi_y(A_1)).
\]
It remains to add two thing together and note that due to \eqref{eq:assumption-A0-A1} we have 
\[
60|A_1|^{1/2}+(2+2\sqrt{2})|A|^{1/2}\leq \left(60\cdot \left(1 - \frac{1}{(1+\sqrt{2})^2} \right)^{1/2} +(2+\sqrt{2})\right)|A|^{1/2}\leq 60|A|^{1/2},
\]
and also
\begin{align*}
6\log{|A_1|} + 1 \leq 6\log{|A|} + 1 - 6\log{\left(\frac{1}{1-\frac{1}{(1+\sqrt{2})^2}}\right)} \leq 6\log{|A|}. 
\end{align*}
\end{proof}
\section{Freiman's theorem}\label{sec:freiman}
In order to deduce a lower bound on $|A+\sqrt{2} A|$ for an arbitrary finite set $A\subset \mathbb{R}$ from Lemma \ref{lm:main-lemma}, we use the following structural theorem due to Green and Ruzsa \cite{Green2007} (the version for $\mathbb{Z}$ is due to Freiman \cite{frejman2008foundations}). To state the result we first recall the definition of a \emph{proper arithmetic progression}.

\begin{definition}
A set $A\subset \mathbb{Z}^2$ is a proper arithmetic progression of dimension $d\geq 1$ if it has the form
\begin{equation}\label{eq:AP}
P=\left\{v_0 +\ell_1v_1+\dots+\ell_d v_d \suchthat 0\leq \ell_j < L_j \right\},
\end{equation}
where $v_0, v_1,\dots, v_d \in \mathbb{Z}^2, L_1, L_2, \dots, L_d\in \mathbb{Z}_+$ and all sums in \eqref{eq:AP} are distinct (in which case $|P|=L_1L_2\dots L_d$).
\end{definition}
\begin{lemma}[Theorem 1.1, \cite{Green2007}]\label{lm:Freiman}
For every $K>0$ there exist constants $d=d(K)$ and $f=f(K)$ such that for any subset $A\subset \mathbb{Z}^2$ with doubling constant at most $K$ (i.e. such that $|A+A|\leq K|A|$) there exists a proper arithmetic progression $P\subset \mathbb{Z}^2$ containing $A$ which has dimension at most $d(K)$ and size at most $f(K)|A|$.
\end{lemma}

\begin{proof}[Proof of Theorem \ref{main}]
By Lemma \ref{lm:reduction} we may assume that $A\subset \mathbb{Z}[\sqrt{2}]$. So the problem is reduced to showing a lower bound on $|A+\T A|$ for an arbitrary finite set $A\subset \mathbb{Z}^2$. We fix an arbitrary set $A\subset \mathbb{Z}^2$ and let $B:=A+\T A$. The idea is to find a non-singular linear transformation $\tau$ commuting with $\T$ such that 
$\tau(\mathbb{Z}^2)\subset \mathbb{Z}^2$ and
for the set $A':=\tau A$
both $\phi_x(A')$ and $\phi_y(A')$ are small
(specifically, small means $O(|A|^\varkappa)$ for certain
$\varkappa<1$).
It allows to apply Lemma \ref{lm:main-lemma} to the set $A'$. The fact that $\tau$ and $\T$ commute ensures that 
\[
B':=A'+\T A' = \tau A + \T \tau A = \tau (A+ \T A) = \tau B.
\]
By Cauchy--Davenport theorem we have 
\begin{equation}\label{eq:two-to-one}
\phi_x(B') = \phi_x(A'+\T A')\geq \phi_x(A')+\phi_y(A')-1,
\end{equation}
so it suffices to choose $\tau$ such that $\phi_x(B') = \phi_x(\tau B)$ is small. We may clearly assume that $|B| = |A+\T A|\leq (1+\sqrt{2})^2|A|$ as otherwise the statement is trivial. Then, by Lemma \ref{lm:doubling}, the set $B=A+\T A$ has doubling constant at most $(1+\sqrt{2})^{12}$ and so by Lemma \ref{lm:Freiman} there exist absolute constants $d, f > 0$ and a proper arithmetic progression 
\begin{equation*}
P=\left\{v_0 +\ell_1v_1+\dots+\ell_d v_d \suchthat 0\leq \ell_j < L_j \right\} \subset \mathbb{Z}^2,
\end{equation*}
such that $B\subset P$ and $L_1L_2\dots L_d = |P|\leq f|B|$. Without loss of generality we may assume that $L_d\geq L_{d-1} \geq\dots \geq L_1$.
Note that $v_d$ has rational coordinates,
thus it
is not an eigenvector of $\T$. Therefore
$v_d$ and $\T v_d$ are linearly independent and 
there exist integers $\alpha,\beta$ such that 
the vector
$\alpha v_d+\beta \T v_d$ is non-zero but has zero abscissa.
Denote 
$\tau:=\alpha \operatorname{Id} + \beta \T$ (it obviously commutes with $\T$ and is not
singular since the eigenvalues of $\T$ are not
rational). 
Then $\tau v_d$ has zero abscissa, and we ensure that 
\[
\phi_x(\tau B) \leq \prod_{j=1}^{d-1} L_j \leq \left(\prod_{j=1}^d L_j\right)^{1-1/d} = |P|^{1-1/d}\leq (f |B|)^{1-1/d}.
\]
Using \eqref{eq:two-to-one} and our assumption that $|B|\leq (1+\sqrt{2})^2 |A|$ we deduce that 
\[
\phi_x(A')+\phi_y(A') \leq 1+(f |B|)^{1-1/d} \leq f_0 |A|^{1-1/d},
\]
where $f_0$ is an absolute constant. It then remains to apply Lemma \ref{lm:main-lemma}  to the set $A'$ (note that $|A'+\T A'| = |A+\T A|$) to see that 
\[
|A+\T A| \geq  (1+\sqrt{2})^2|A| - 60|A|^{1/2} - 6f_0 |A|^{1-1/d}\cdot \log{|A|}  \geq (1+\sqrt{2})^2|A| - C\cdot |A|^{1-\eps},
\]
with some $\eps < 1/d$ and absolute constant $C$ large enough. 
\end{proof}

\section{$A+\T A$ in continuous setting}\label{sec:continuous-setting}

Let $\mu$ denote the 
Lebesgue measure in $\mathbb{R}^d$, and 
let the lower $*$ denote the inner measure
(so, $\mu_*$ is the inner Lebesgue measure
in $\mathbb{R}^d$). 
For
a linear operator $\T\in \End(\mathbb{R}^d)$
denote
$$
H(\T)=\prod_{i=1}^d(1+|\lambda_i|)
$$
where $\lambda_1,\ldots,\lambda_d$
are the (complex) eigenvalues of $\T$,
listed with algebraic multiplicities.

\begin{theorem}\label{th:continuous}
Let $\T\in \End(\mathbb{R}^d)$
be a linear operator.
Then for any 
set $K\subset \mathbb{R}^d$ we have
\begin{equation}\label{continuous}
    \mu_\star (K+\T K)\ge H(\T) \cdot \mu_\star (K).
\end{equation}
\end{theorem}

\begin{proof}
In what follows we assume that $K$
is compact. (The general case follows by passing to
a limit. Indeed,
$K$ contains compact subsets $K_1,K_2,\ldots$
such that $\mu_\star(K)=\lim \mu(K_n)$,
and $K+\T K$ contains $K_n+\T K_n$ that
yields $\mu_\star (K+\T K)\ge \mu(K_n+\T K_n)$. 
So, passing to a limit in \eqref{continuous} 
for $K_n$ we get it for $K$.)

Next, we assume that all $\lambda_i$'s
are distinct. Again, the general case follows by a limit procedure. Indeed, there exist
open neighborhoods $U_1,U_2,\ldots$ of $K+\T K$
such that $\mu(K+\T K)=\lim \mu(U_n)$. For each
$n$ there exists an operator $\T_n$
with distinct eigenvalues which is so close to $\T$
that $K+\T_nK\subset U_n$ and therefore 
$\mu(U_n)\ge \mu(K+\T_nK)$. We may also
assume that $\|\T-\T_n\|\to 0$. 
Thus if \eqref{continuous} holds for
$\T_n$, it also holds for $\T$ (we use here the
well-known fact that the spectrum of the limit of operators in $\mathbb{R}^d$
equals to the limit of their spectra.)

If $\lambda_1$ is real, let $E_1$ be an eigenspace
of $\lambda_1$; if $\lambda_1$ is not real,
and $u+iv (u,v\in \mathbb{R}^d)$ is a corresponding complex eigenvector, let $E_1$ be a span of $u,v$.
This allows to write $\mathbb{R}^d=E_1\oplus E_2$,
where $E_1,E_2$ are $\T$-invariant linear subspaces,
and either 

(i) $\dim E_1=1$; or 

(ii) $\dim E_1=2$
and $\T$ acts on $E_1$ as a rotational homothety.

If $E_2=\{0\}$, then \eqref{continuous}
follows directly from the 1- or 2-dimensional Brunn--Minkowski inequality. So, using induction
we may suppose that $\dim E_2>0$,
and the restriction of $\T$ to $E_2$ 
satisfies \eqref{continuous}.

For $z\in E_2$ denote $K_z=\{x_1\in E_1\suchthat
x_1+z\in K\}$, $f(z)=\nu_1(K_z)$, where $\nu_i$
is the Lebesgue measure in $E_i$, $i=1,2$,
and without loss of generality $d\mu(x_1+x_2)=d\nu_1(x_1)d\nu_2(x_2)$ for $x_1\in E_1,x_2\in E_2$.
Next, denote $a=\sup(f)$ and for 
$t\in (0,1)$ denote $X(t)=\{x\in E_2\suchthat f(x)\ge ta\}$.
The sets $X(t)$ are measurable
(since $f$ is measurable by Fubini theorem)
and non-empty and we have
\begin{equation*}
    \mu(K)=\int_{E_2} f(x)d\nu_2(x)=
    \int_0^\infty \nu_2\{x\in E_2\suchthat
    f(x)\ge \tau\}d\tau=a\int_0^1 \nu_2(X(t))dt.
\end{equation*}
Choose $x,y\in X(t)$. Note that
$$
x+\T y+K_x+\T K_y\subset K+\T K=:L
$$
By 1- or 2-dimensional Brunn--Minkowski inequality
for non-empty compact sets $K_x,K_y\subset E_1$
we have 
$$\nu_1(K_x+\T K_y)\ge H\left(\T|_{E_1}\right) \min(\nu_1(K_x),\nu_1(K_y))\ge 
ta\cdot H\left(\T|_{E_1}\right).
$$
Therefore, if we use the notation $L_z := \{x_1\in E_1\suchthat
x_1+z\in L\}$, we have
$$
\nu_2\left(z\in E_2\suchthat \nu_1(L_z)\ge
ta\cdot H\left(\T|_{E_1}\right) \right)
\ge \nu_{2*}\left(X(t)+\T X(t)\right)\ge
\nu_{2*} (X(t))\cdot H\left(\T|_{E_2}\right) 
$$
by the induction hypothesis. Therefore
\begin{align*}
    \mu(L)&=\int_{0}^\infty 
\nu_2\left(z\in E_2\suchthat \nu_1(L_z) \ge
\tau\right)d\tau\\&\ge 
a\cdot H\left(\T|_{E_1}\right) 
\int_{0}^1 \nu_2\left(z\in E_2\suchthat \nu_1(L_z)\ge
ta\cdot H\left(\T|_{E_1}\right) \right)dt\\
&\ge a\cdot H\left(\T|_{E_1}\right)
H\left(\T|_{E_2}\right)\int_0^1 
\nu_{2*} (X(t))dt =H(\T) \mu(K).
\end{align*}
\end{proof}

\begin{remark}
The bound $H(\T)$ in Theorem \ref{th:continuous}
is sharp. If $\T$ is complex diagonalizable, 
we may find a convex compact set $K=K(\T)\subset\mathbb{R}^d$ such that
\eqref{continuous} turns into equality. 
It suffices to decompose $\mathbb{R}^d$
as a direct sum of 1- and 2-dimensional
$\T$-invariant subspaces, onto each of which $\T$
acts as a rotational homothety, and take $K$
to be the direct 
product of balls in these subspaces. 

In the general case we fix a complex
diagonalizable operator $\T_0$ with the same
spectrum as $\T$, then find a convex
compact set $K(\T_0)$, then find a sequence
of operators $\T_n\to \T_0$ which are similar
to $\T$: $\T_n=S_n \T S_n^{-1}$.
The sets $K_n:=S_n^{-1} K(\T_0)$
satisfy \begin{align*}
    \frac{\mu(K_n+\T K_n)}
{\mu(K_n)}&=\frac{\mu(S_nK_n+S_n\T K_n)}
{\mu(S_nK_n)}=\frac{\mu(K(\T_0)+\T_nK(\T_0))}
{\mu(K(\T_0))}\\&\to \frac{\mu(K(\T_0)+T_0K(\T_0))}
{\mu(K(\T_0))}=H(\T_0)=H(\T).
\end{align*}
\end{remark}

Now we formulate a
general conjecture on $|A+\alpha A|$
for algebraic $\alpha$.

For an irreducible polynomial $f(x)\in \mathbb{Z}[x]$ of degree
$d\geqslant 1$
(irreducibility in particular means that
the coefficients of $f$ do not have a common
integer divisor greater than 1) denote
$$H(f)=\prod_{i=1}^d (|a_i|+|b_i|),$$
where $f(x)=\prod_{i=1}^d (a_ix+b_i)$
is a full complex factorization of $f$
(clearly the value $H(f)$ is well-defined).

\begin{proposition}\label{upper}
Let $\alpha$ be an algebraic real number with minimal polynomial $f(x)\in \mathbb{Z}[x]$
Then
\begin{equation}\label{eq:algebraic}
    \lim_{n\rightarrow \infty} \min_{A\subset \mathbb{R}, |A| = n} \frac{|A+\alpha A|}{|A|} \le  H(f).
\end{equation}
\end{proposition}

\begin{proof}
The upper bound follows from the sharpness
of the continuous bound \eqref{continuous}
for convex compact sets, see Remark after Theorem
\ref{th:continuous}. 
Namely, consider the following operator
$\T$: $g\mapsto x\cdot g$ 
in the $d$-dimensional factor space 
$\mathbb{R}[x]/f(x)\mathbb{R}[x]$. 
Let $\{1,x,\ldots,x^{d-1}\}$ be the standard
basis in this space, and normalize
Lebesgue measure appropriately.

It is not hard to see that $H(f)=|c|\cdot H(\T)$,
where $c$ is the leading coefficient of $f$. 
Choose a convex compact set $K\subset \mathbb{R}^d$ such that
$\mu(K+\T K)/\mu(K)=H(\T)$ 
(the eigenvalues
of $\T$ are $d$ algebraic conjugates of $\alpha$
and they are distinct, since $f$ is irreducible
and therefore $f$ and $f'$ are coprime;
thus $\T$ is diagonalizable, and such 
$K$ exists due to Remark after 
Theorem \ref{th:continuous}.)
Then take large $M>0$ 
and consider the set
$$
\Omega_M=\{a_0+a_1x+\ldots+a_{d-1} x^{d-1}\in M\cdot K\suchthat 
a_i\in \mathbb{Z}, c|a_{d-1}\}.
$$
The number of such points is
$|\Omega_M|=|c|^{-1} M^d\mu(K)+o(M^d)$. 
On the other hand, all points in $\T\Omega_M$
have integer coordinates (that's why we 
required $c|a_{d-1}$). Therefore
$|\Omega_M+\T\Omega_M|\le M^d \mu(K+\T K)+o(M^d)$.
Finally $$\frac{|\Omega_M+T\Omega_M|}{|\Omega_M|}\le
|c| \frac{\mu(K+\T K)}{\mu(K)}+o(1)=
|c|H(\T)=H(f).$$
It remains to take $A=\{g(\alpha)\suchthat
g(x)\in \Omega_M\}$.
\end{proof}

\begin{conjecture}\label{conj:1}
For any real algebraic $\alpha$
the inequality \eqref{eq:algebraic}
turns into equality. For complex algebraic $\alpha$
the analogous equality holds for $A\subset \mathbb{C}$.
\end{conjecture}

This conjecture is a partial case of the following

\begin{conjecture}\label{conj:2}
Let  $\mathcal{T}:\mathbb{R}^d\rightarrow\mathbb{R}^d$ be a linear operator with a characteristic polynomial $f(\mathcal{T})$, then 
\[
\lim_{n\rightarrow \infty} \min_{A\subset \mathbb{R}^d, |A| = n} \frac{|A+\mathcal{T} A|}{|A|} = \min_{g| f(\mathcal{T}), g\in \mathbb{Z}[x]} 
H(g),
\]
where the minimum is taken over
all irreducible divisors $g$ of $f(\mathcal{T})$.
If
there are no polynomials $g$ with rational coefficients that divide $f(\mathcal{T})$, we define the minimum to be infinity. 
\end{conjecture}

It is interesting to find the analogue of
Theorem \ref{th:continuous} and
Conjectures \ref{conj:1}, \ref{conj:2} for several operators. In this direction we recall a conjecture by Boris Bukh \cite[problem 5]{BBwebpage}: 
\begin{equation}\label{BB}
\sum_{i=1}^k |\mathcal{T}_i A|\ge 
\left(\sum_i |\det \mathcal{T}_i|^{1/d} -o(1)\right)
|A|\end{equation}
for large sets $A\subset \mathbb{Z}^d$, where $\mathcal{T}_i$ are $k$
linear operators preserving $\mathbb{Z}^d$
without common invariant subspace such that
$\sum_i \mathcal{T}_i \mathbb{Z}^d=\mathbb{Z}^d$. The continuous analogue of \eqref{BB} immediately
follows from Brunn--Minkowski inequality,
and is in general not tight even for $k=2$
as follows from Theorem \ref{th:continuous}.

\medskip
We are grateful to B.~Bukh  and to I.~Shkredov and S.~Konyagin for drawing our attention
to \cite{BBwebpage} and \cite{sanders2012,schoen2011} respectively.

\bibliographystyle{siam}
\bibliography{bibliography1}

\end{document}